\documentclass[11pt]{article}
\usepackage{amsmath}
\usepackage{amsthm}
\usepackage{amssymb}
\usepackage{amscd}
\usepackage[dvipdfm]{hyperref}

\newcommand{\R}{\mathbb{R}}

\theoremstyle{plain}
\newtheorem{theorem}{Theorem}[section]
\newtheorem{proposition}[theorem]{Proposition}

\newtheorem{Corollary}[theorem]{Corollary}

\theoremstyle{definition}
\newtheorem{definition}[theorem]{Definition}

\makeatletter
\@addtoreset{equation}{section}
\makeatother
\newcounter{constantLABEL}
\newcommand{\cref}[1]{C_{\ref{#1}}}

\newcounter{constantslabel}
\begin{document}


\title{Stable sheaves with twisted sections and 
the Vafa--Witten equations on smooth projective surfaces} 


\author{Yuuji Tanaka}
\date{}


\maketitle


\begin{abstract}
This article describes a Hitchin--Kobayashi style correspondence for the
 Vafa--Witten equations on smooth projective surfaces. 
This is an equivalence between a suitable notion of stability for a pair
 $(\mathcal{E}, \varphi)$, where $\mathcal{E}$ is a locally-free sheaf
 over a surface $X$ and $\varphi$ is a section of $\text{End}
 (\mathcal{E}) \otimes K_{X}$; and the existence of a solution to
 certain gauge-theoretic equations, the Vafa--Witten equations, for a
 Hermitian metric on $\mathcal{E}$.  
It turns out to be a special case of results obtained by 
\'{A}lvarez-C\'{o}nsul and Garc\'{\i}a-Prada on the quiver vortex equation. 
In this article, we give an alternative proof which uses a
 Mehta--Ramanathan style argument 
originally developed by Donaldson for the Hermitian--Einstein problem, 
as it relates the subject with the Hitchin equations on Riemann
 surfaces, and surely indicates a similar proof of the existence of a
 solution under the assumption of stability 
 for the Donaldson--Thomas instanton equations described in 
\cite{tanaka2} on smooth projective
 threefolds; and more broadly that for the quiver vortex equation 
on higher dimensional
 smooth projective varieties.   
\end{abstract}



\section{Introduction}

In this article, we consider a set of gauge-theoretic equations on smooth
projective surfaces, introduced  
by Vafa and Witten \cite{VW} in the study of {\it $S$-duality} conjecture for 
$N=4$ supersymmetric Yang--Mills theory originally on closed four-manifolds, 
and recently discussed also by Haydys \cite{Ha} and \cite{W}
in the context of ``categorification'' of Khovanov homology.

The equation can be seen as a higher-dimensional
analogue of the Hitchin equation on compact Riemann surfaces
\cite{Hit}. 
The Hitchin equation is an equation for a pair consisting of a holomorphic
structure on a vector bundle $E$ over a Riemann surface $\Sigma$, 
and a holomorphic section $\Phi$ of the associated bundle $\text{End}\, (E) \otimes
K_{\Sigma}$, where $K_{\Sigma}$ is the canonical bundle of $\Sigma$.  
Simpson \cite{Si} generalized it to higher dimensions for a pair
$(\mathcal{E}, \theta)$, 
where $\mathcal{E}$ is a torsion-free sheaf on a projective variety $X$, 
and $\theta$ is a section of $\text{End} \, (\mathcal{E}) \otimes
\Omega^{1}_{X}$.   
The Vafa--Witten equation can be seen as an analogue of the Hitchin
equations for surfaces, but in a different way of that pursued by Simpson
mentioned above,
since it takes up a section of 
$\text{End}\, (E) \otimes K_{X}$ 
as an extra field, 
which is just the same as in the Hitchin case,    
 rather than that of $\text{End}\, (E) \otimes
\Omega^{1}_{X}$ as in the Simpson case.  
Also, the Donaldson--Thomas instanton equation on compact K\"{a}hler
threefolds, described in
\cite{tanaka2}, can be seen as a three-dimensional counterpart of the
Hitchin equation  
in the same way as the Vafa--Witten equations. 
More broadly, these equations can be seen as special cases of those
studied by \'{A}lvarez-C\'{o}nsul and Garc\'{\i}a-Prada \cite{AG} 
as the case of a twisted quiver bundle with one vertex and one arrow, 
whose head and tail coincide, and with twisting sheaf the anti-canonical
bundle.

\paragraph{The Vafa--Witten equations.}

Let us describe the equation in the original form first. 
Let $X$ be a closed, oriented, smooth Riemannian four-manifold with Riemannian
metric $g$, and let $P \to X$ be a principal $G$-bundle over $X$
with $G$ being a compact Lie group. 
We denote by $\mathcal{A}_{P}$ the set of all connections of $P$, 
and by $\Omega^{+} (X, \mathfrak{g}_{P})$ the set of self-dual two-forms
valued in the adjoint bundle $\mathfrak{g}_{P}$ of $P$. 
We consider the following equations for a triple $(A, B , \Gamma) \in
\mathcal{A}_{P} 
\times \Omega^{+} (X, \mathfrak{g}_{P}) \times \Omega^{0} (X, 
\mathfrak{g}_{P})$. 
\begin{gather}
d_{A} \Gamma + d_{A}^{*} B = 0, 
\label{VW1}\\ 
F_{A}^{+}  + [B.B ] +  [ B , \Gamma] =0 , 
\label{VW2}
\end{gather}
where $F_{A}^{+}$ is the self-dual part of the curvature of $A$, 
and $[B. B ] \in \Omega^{+} (X , \mathfrak{g}_{P})$ 
(See \cite[\S A.1]{BM}, or \cite[\S 2]{tanaka6} for its definition). 
We call these equations the {\it Vafa--Witten equations}. 
The above equations \eqref{VW1} and \eqref{VW2} with a gauge fixing
condition form an elliptic system with the index being always zero.

Mares studied analytic aspects of the Vafa--Witten equations
in his Ph.D thesis \cite{BM}. 
He also described the equations on compact K\"{a}hler surfaces, and
discussed a relation between the existence of a solution to the
equations 
and a stability of vector bundles as
mentioned below.

\paragraph{The equations on compact K\"{a}hler surfaces.}

Let $X$ be a compact K\"{a}hler surface, and let $E$ a Hermitian vector
bundle of rank $r$ over $X$. 
On a compact K\"ahler surface, the Vafa--Witten equations \eqref{VW1}
and \eqref{VW2} reduce to the following (see \cite[Chap.7]{BM} for the detail). 
\begin{gather*}
 \bar{\partial}_{A} \varphi = 0 , \\
F_{A}^{0,2} = 0 , \quad F_{A}^{1,1} \wedge \omega +  [ \varphi , 
 \bar{\varphi}] = \frac{i \lambda (E)}{2} Id_{E} \, \omega^2 ,
\end{gather*}
where $\varphi \in \Omega^{2,0} (X, \text{End}(E))$, and 
$\lambda (E) = 2 \pi c_{1} (E) \cdot [\omega] / r [ \omega ]^2$.

As we mentioned above, 
this can be seen as a generalization of the Hitchin equation \cite{Hit} 
to K\"{a}hler surfaces. 
In fact, the stability condition that we consider is an analogy of that to
the Hitchin equation.

\paragraph{Stability for pairs.}

We consider a pair $( \mathcal{E} , \varphi)$
consisting of a torsion-free sheaf $\mathcal{E}$ 
and a section $\varphi$ of $\text{End} (\mathcal{E}) \otimes K_{X}$, 
which satisfies a stability condition.  
The stability here is defined by a slope for $\varphi$-invariant
subsheaves similar to the Hitchin equation case \cite{Hit}.

Let $X$ be a compact K\"{a}hler surface,
and let $\mathcal{E}$ be a torsion-free sheaf on $X$, and 
let $\varphi$ be a section of $\text{End} 
(\mathcal{E} )\otimes K_{X}$.  
A subsheaf $\mathcal{F}$ of $\mathcal{E}$ 
is said to be a {\it $\varphi$-invariant} if 
$ \varphi  (\mathcal{F}) \subset \mathcal{F} \otimes K_{X}$.   
We define a {\it slope} $\mu (\mathcal{F})$ of a coherent subsheaf $\mathcal{F}$ of
$\mathcal{E}$ by 
$ \mu (\mathcal{F}) 
:= \frac{1}{\text{rank} ( \mathcal{F} )} \int_{X} 
c_1 ( \det \mathcal{F}) \wedge \omega $.

\begin{definition}
A pair $(\mathcal{E},\varphi)$ 
consisting of a torsion-free sheaf $\mathcal{E}$ 
and a section $\varphi$ of $\text{End} (\mathcal{E}) \otimes K_{X}$ 
is called {\it semi-stable} if 
$ \mu ( \mathcal{F} ) \leq \mu ( \mathcal{E} ) $ 
for any 
$\varphi$-invariant coherent subsheaf 
$\mathcal{F}$ with $\text{rank} (\mathcal{F}) 
< \text{rank} (\mathcal{E})$.  
A pair $(\mathcal{E}, \varphi)$ is called {\it stable} if 
$ \mu ( \mathcal{F} ) < \mu ( \mathcal{E} ) $ 
for any 
$\varphi$-invariant coherent subsheaf 
$\mathcal{F}$ with $\text{rank} (\mathcal{F}) 
< \text{rank} (\mathcal{E})$.  
\label{def:stable}
\end{definition}

\begin{definition}
A pair $(\mathcal{E} , \varphi)$ 
consisting of a torsion-free sheaf $\mathcal{E}$ 
and a section $\varphi$ of $\text{End} (\mathcal{E}) \otimes K_{X}$ 
is said to be {\it poly-stable}
if it is a direct sum of stable sheaves with the same slopes in the sense of Definition
\ref{def:stable}. 
\label{def:pstable}
\end{definition}

\paragraph{The Hitchin--Kobayashi correspondence for the Vafa-Witten equations.}

A correspondence we describe in this article is a one-to-one
correspondence between  the existence of 
a solution to the Vafa--Witten equations on  a
locally-free 
sheaf $\mathcal{E}$ on a smooth projective surface $X$ and the stability
in the sense of Definition \ref{def:stable}. 
This fits into the setting for the above mentioned twisted 
quiver bundles and the quiver vortex equation studied by 
\'{A}lvarez-C\'{o}nsul and Garc\'{\i}a-Prada \cite{AG} (see also
\cite{BGM}), 
and the correspondence turns out to be a special case of their results. 
We state it in our setting as follows.

\begin{theorem}[\cite{AG}]
Let $X$ be a K\"{a}hler surface with K\"{a}hler form $\omega$. 
Let $(\mathcal{E} , \varphi)$ be a pair consisting of a locally-free
 sheaf $\mathcal{E}$ on $X$ and a section $\varphi \in \text{End}\, (\mathcal{E})
 \otimes K_{X}$, where $K_{X}$ is the canonical bundle of $X$. 
Then, $(\mathcal{E} , \varphi)$ is poly-stable if and only if 
$\mathcal{E}$ admits a unique Hermitian metric $h$ satisfying 
$F_{h} + \Lambda [ \varphi , \bar{\varphi}^{h}] = i \frac{\lambda
 (E)}{2} Id_{E} \omega$, where $F_{h}$ is the curvature form of $h$, and
 $\Lambda := ( \wedge \omega )^{*}$. 
\end{theorem}

In this article, we give an alternative proof for the existence part of
 the above theorem in the case of smooth projective surfaces, stated
 below as Theorem \ref{ith:exist}, 
by using a Mehta--Ramanathan style theorem.

\begin{theorem}
Let $X$ be a smooth projective surface, and let $E$ be a holomorphic vector bundle on
 $X$. 
We take a holomorphic section $\varphi$ of $\text{End{(E)}} \otimes
 K_{X}$, 
where $K_{X}$ is the canonical bundle of $X$. 
We assume that $(E , \varphi)$ is poly-stable in the sense of Definition
 \ref{def:pstable} 
Then there exists a unique Hermitian metric $h$ of $E$   
 such that the equation $F_{h} + \Lambda  [ \varphi, \bar{\varphi}^{h}]  = 
 \frac{i \lambda (E)}{2} Id_{E} \, \omega$ is satisfied. 
\label{ith:exist}
\end{theorem}

Our proof, given in the next section, also uses a Donaldson-type  
functional on the space of Hermitian metrics on $E$, 
which is a modification of that defined by Donaldson in \cite{D} for
solving the Hermitian--Einstein problem. 
As in the Hermitian--Einstein case, 
one main point is to obtain a lower
bound for the functional. 
To achieve this we use 
a Mehta--Ramanathan style argument, 
in other words, we reduce 
the problem to the corresponding lower dimensional
one; the Hitchin equation on compact Riemann surfaces in our case. 
This method was developed originally by 
Donaldson \cite{D} to prove the existence of a
 Hermitian--Einstein metric on a holomorphic vector bundle over a smooth
 projective surface under the assumption of stability.

We remark that 
a parallel argument to that described in this article 
should give an alternative proof of the existence of a
 solution under the assumption of stability 
 for the Donaldson--Thomas instanton equations described in 
\cite{tanaka2} on smooth projective
 threefolds; and more broadly that for the quiver vortex equation 
on higher dimensional 
 smooth projective varieties.

\paragraph{Acknowledgements.}
I would like to thank Kotaro Kawatani for useful conversations around the subject. 
I would also like to thank Oscar Garc\'{\i}a-Prada for informing me of
the results in \cite{AG} and \cite{BGM}. 
I am grateful to a referee for 
many useful comments and helpful suggestions.

\section{Existence of a solution under the assumption of stability}

In Section \ref{sec:func}, 
we introduce a functional on the space of Hermitian metrics of a
holomorphic vector bundle over a compact K\"{a}hler surface, 
and lay out some properties of the functional such as its
critical points and convexity. 
In Section \ref{sec:Hit}, we mention the Hitchin equation on a compact
Riemann surface and its generalization. 
We then give a proof of Theorem \ref{ith:exist} 
using a Mehta--Ramanathan style theorem in Section \ref{sec:exist}.

\subsection{Functional and its properties}
\label{sec:func}

Let $E$ be a holomorphic vector bundle over a compact K\"{a}hler
surface $X$. 
We denote by $\text{Herm}^{+} (E)$ the set of all $C^{\infty}$ Hermitian metrics
on $E$.  
Let $k ,h \in \text{Herm}^{+} (E)$. 
We connect them by a smooth curve $h_{t} \, (0 \leq t \leq 1)$ so that 
$k = h_{0}$ and $h = h_{1}$.     
We denote by $F_{h_{t}} = F (h_{t}) \in A^{1,1} (\text{End} (E) )$ 
the curvature of $h_{t}$. 
We put $v_{t} = h_{t} \partial_{t} h_{t} \in A^{0} (\text{End} (E))$. 
Let $\varphi$ be a holomorphic section of $\text{End} (E) \otimes K_{X}$, 
where $K_{X}$ is the canonical bundle of $X$. 
We define the following. 
\begin{gather*}
 Q_{1} (h,k) := \log ( \det (k^{-1} h)),  \\
 Q_2 (h , k) := \sqrt{-1} \int_{0}^{1} 
 \text{tr} \, ( v_{t} \cdot (F_{h_{t}} + \Lambda  [ \varphi ,
 \bar{\varphi}^{h_{t}}]  )) dt.
\end{gather*}
We then consider the following functional for pairs of Hermitian metrics
$h$ and $k$.    
\begin{equation*}
 D_{\varphi} (h,k) :=  \int_{X} Q_{2} (h ,k)
\wedge \omega 
 - \frac{\lambda (E)}{2} \int_{X} 
\, Q_{1} (h,k) dV_{g},  
\end{equation*}
where $\lambda (E) = 2 \pi  c_1
(E) \cdot [\omega]   / r [\omega]^2 $. 
This is a modification of the functional defined by Donaldson in
\cite{D} for solving the Hermitian--Einstein problem.

Firstly, we prove that $D_{\varphi} (h,k)$ does not
depend upon the choice of a curve joining $k$ and $h$. 
As in the case of the Hermitian--Einstein metrics, 
one can prove the following. 
\begin{proposition}
Let $h_{t}\, (a \leq t \leq b)$ be a differentiable curve in
 $\text{Herm}^{+} (E)$, 
and let $k$ be a fixed Hermitian metric of $E$. 
Then 
\begin{equation}
 \sqrt{-1} \int_{a}^{b} \text{tr}\, \left( v_t \cdot \left( F_{h_{t}} + 
\Lambda [ \varphi ,\bar{\varphi}^{h_{t}}] \right) \right) dt 
+ Q_2 (h_{a} ,k) 
-Q_2 (h_{b} , k)  
\label{eq:lem3.2}
\end{equation}
lies in 
$\partial A^{0,1} + \bar{\partial} A^{1,0}$. 
\label{prop:df1}
\end{proposition}

\begin{proof}
Our  proof goes in a similar way to that by Kobayashi \cite[Chap.~VI,
 Lem.~3.6]{Ko} 
for the Hermitian--Einstein metrics   
except that we deal with the extra field $\varphi$. 
Let $\Delta$ be the domain in $\R^2$ defined by 
$ \Delta := \{ (t ,s) \, : a \leq t \leq b \, , \, 0 \leq s \leq 1 \}$, 
$h : \Delta \to \text{Herm}^{+} (E)$ a smooth map with $h (t , 0) = k , 
h(t , 1) =h_{t}$ for $a \leq t \leq b$, 
and $h(a,s)$ and $h(b,s)$ are the line segments from $k$ to $h_{a}$ and 
$k$ to $h_{b}$ respectively. 
Set $v = h^{-1} \partial_{s} h , w = h^{-1} \partial_{t} h, 
F = \bar{\partial} (h^{-1} \partial h)$, and 
$\Phi = \sqrt{-1} \text{tr} \left( 
h^{-1} \tilde{d} h ( F + \Lambda [\varphi , \bar{\varphi}^{h}]) \right)$, 
where $\tilde{d} = \frac{\partial}{\partial s}  ds +
 \frac{\partial}{\partial t} dt$, the exterior derivative on $\Delta$. 
We then use the Stokes formula for the 1-form $\Phi$, namely, 
\begin{equation}
\int_{\Delta} \tilde{d} \Phi 
 = \int_{\partial \Delta} \Phi .
\label{eq:stokes}
\end{equation}
The right-hand-side of \eqref{eq:stokes} becomes  
$$ \int_{\partial \Delta} \Phi 
= \sqrt{-1} \int_{a}^{b} \text{tr} 
\left( v_{t} \cdot  ( F_{h_{t}} + 
\Lambda [ \varphi ,\bar{\varphi}^{h_t}]  )\right)  dt 
+ Q_2 (h_{a} ,k) 
-Q_2 (h_{b} , k) . 
$$  
Therefore we need to prove 
that $\tilde{d} \Phi \in \partial A^{0,1} + \bar{\partial} A^{1 ,0}$. 
Put $M= F + \Lambda [ \varphi , \bar{\varphi}^{h}]$. 
From the definition of $\Phi$, we get 
$$ \tilde{d} \Phi 
= \sqrt{-1} \text{tr} 
\left(  ( \partial_{s} v - \partial_{t} w ) M 
 - w \partial_{t} M + v \partial_{s} M \right) ds \wedge dt . $$
Furthermore, some calculations similar to \cite[pp.199--200]{Ko} show that 
$ \partial_{s} v  = - w v + h^{-1} \partial_{s} \partial_{t} h ,
\partial_{t} w = - v w + h^{-1} \partial_{t} \partial_{s} h $, 
$\partial_{t} M = \bar{\partial} D' v 
+ \Lambda [ \varphi, \partial_{t} \bar{\varphi}^{h}]$, 
$\partial_{s} M = \bar{\partial} D' w 
+ \Lambda [ \varphi , \partial_{s} \bar{\varphi}^{h}]$, 
where $D = D' + D'' (= D' + \bar{\partial})$ is the exterior covariant
 differentiation of the Hermitian connection defined by $h$.  
Using these, we get 
\begin{equation}
\begin{split}
\tilde{d} \Phi 
 &= \sqrt{-1} \text{tr} \left( 
 ( v w - wv ) F - w \bar{\partial} D' v + 
v \bar{\partial} D' w \right) ds \wedge dt \\
 &+ \sqrt{-1} \text{tr} \, \left( \Lambda 
 ( v [\varphi , \partial_{s} \bar{\varphi}^{h} ] 
     - w [ \varphi , \partial_{t} \bar{\varphi}^{h}] 
   + (vw -w v) [ \varphi , \bar{\varphi}^{h}] ) \right) ds \wedge dt . \\
\end{split}
\label{eq:dphi}
\end{equation}
One can easily check that the second term 
of \eqref{eq:dphi} vanishes because of  $\partial_{s} \bar{\varphi}^{h} 
= [ \bar{\varphi}^{h} , w ]$, $\partial_{t} \bar{\varphi}^{h} = [ \bar{\varphi}^{h} ,v]$,  
and the Jacobi identity. 
On the other hand, the first term of \eqref{eq:dphi}, 
which does not involve the extra field $\varphi$,  
 is the same term in the Hermitian--Einstein case as in \cite{Ko}, 
and it becomes 
$$  - \sqrt{-1} \text{tr} \, 
\left( v D' \bar{\partial} w + w \bar{\partial} D' v \right) ds \wedge dt .$$
Hence, defining the $(0,1)$-form $\alpha := \sqrt{-1} \text{tr} \, ( v \bar{\partial} w)$, 
we get 
$$ \tilde{d} \Phi = - 
\left( \partial \alpha + \bar{\partial} \bar{\alpha} 
+ \sqrt{-1} \bar{\partial } \partial \text{tr} \, (vw )\right) ds \wedge dt .$$ 
Thus, $\tilde{d} \Phi \in \partial  A^{0,1} + \bar{\partial} A^{1,0}$. 
\end{proof}

From Proposition \ref{prop:df1}, 
we deduce the following. 
\begin{Corollary}
Let $h_{t} \, (a \leq t \leq b)$ be a piecewise differentiable closed 
curve in $\text{Herm}^{+} (E)$ (namely, $h_{a} = h_{b}$). 
Put $v_{t} = h_{t}^{-1} \partial_{t} h_{t}$.  
Then 
\begin{equation*}
 \sqrt{-1} \int_{a}^{b} \text{tr}\, \left( v_t \cdot \left( F_{h_{t}} + 
\Lambda [ \varphi ,\bar{\varphi}^{h_{t}}] \right) \right) dt 
\end{equation*} 
lies in $\partial A^{0,1} + \bar{\partial} A^{1,0}$.  
\label{lem2}
\end{Corollary}

Hence we obtain the following.  
\begin{proposition}
$D_{\varphi} (h , k)$ does not depend on the choice of a curve joining $k$ to $h$. 
\end{proposition}

We next fix a Hermitian metric $k$ on $E$, and define a functional 
$D_{\varphi} : \text{Herm}^{+} (E) \to \R$ by 
$D_{\varphi} (h) :=  D_{\varphi} (k ,h)$ for $h \in \text{Herm}^{+} (E)$. 
Following \cite[Chap.VI \S 3]{Ko}, one can prove the following two
propositions. 
The first says that the critical points of the functional are solutions
of the Vafa--Witten equations.

\begin{proposition}
Let $k$ be a fixed Hermitian metrics on $E$. 
Then $h$ is a critical point of $D_{\varphi} (\cdot) := D_{\varphi} (k, \cdot)$ if and only if 
$h$ satisfies $F_{h} + \Lambda [\varphi , \bar{\varphi}^{h}] = 
\sqrt{-1} \frac{\lambda (E)}{2} Id_{E} \omega$. 
\label{prop:crit} 
\end{proposition}

\begin{proof}
Let $h_{t} \, (a \leq t \leq b)$ be a differentiable curve in
 $\text{Herm}^{+} (E)$, which connects $h$ and $k$. 
Then, 
differentiating \eqref{eq:lem3.2} with respect to $t$, 
we get
$$ \frac{d}{dt} Q_2 (h_{t} ,k) 
 = \sqrt{-1} \text{tr} (v_{t} \cdot (F_{h_{t}} + \Lambda  
 [\varphi, \bar{\varphi}^{h_{t}} ] ))  $$
up to $\partial A^{0,1} + \bar{\partial} A^{1,0}$. 
In addition, we have $\partial_{t} Q_{1} (h_{t} , k) = \text{tr} (v_{t})$. 
Hence we obtain 
\begin{equation}
\frac{d}{dt} D_{\varphi} (h_t) =  \sqrt{-1} \int_{X} \text{tr} \,  
( v_t \cdot\mu_{\varphi} (h_t)) , 
\label{eq:fd} 
\end{equation}
where 
$\mu_{\varphi} (h_t) 
:= F_{h_t} \wedge \omega + [ \varphi , \bar{\varphi}^{h_t}] 
- \sqrt{-1} \frac{\lambda (E)}{2} Id_{E} 
\omega^2$. 
Thus the assertion holds. 
\end{proof}

The next proposition says that the functional $D_{\varphi} ( \cdot )$ is
convex. 
\begin{proposition}
Let $k$ be a fixed Hermitian metric on $E$, and let $\tilde{h}$ be a critical point
 of $D_{ \varphi} ( \cdot) = D_{ \varphi} (k , \cdot)$. 
Then $D_{ \varphi} (\cdot)$ attains an absolute minimum at
 $\tilde{h}$. 
\label{prop:convex} 
\end{proposition}

\begin{proof}
Let $h_{t} \, (0 \leq t \leq 1)$ be a differential curve with $h_{0} = \tilde{h}$. 
Differentiating \eqref{eq:fd}, we get 
\begin{equation*}
\frac{d^2}{dt^2} D_{ \varphi} (h_t) 
= \frac{d}{dt} \sqrt{-1} \int_{X} \text{tr}\, (v_t \cdot \mu_{ \varphi} (h_{t})) 
= \sqrt{-1} \int_{X} \left( (\partial_{t} v_t) \cdot \mu_{\varphi} + v_t \cdot
 \partial_{t} \mu_{\varphi} \right). 
\end{equation*}
Furthermore, since 
$\partial_{t} \bar{\varphi}^{h} = [ \bar{\varphi}^{h} , v]$, 
we obtain 
\begin{equation}
\partial_{t} \mu_{\varphi} = \bar{\partial} D' v\wedge \omega 
+ [ \varphi , [ \bar{\varphi}^{h_t} , v]]. 
\label{eq:dtmu}
\end{equation} 
As $h_{0}$ is a critical point of $D_{\varphi} ( \cdot)$, we get 
\begin{equation*}
\begin{split}
\left. 
\frac{d^2}{dt^2}
 D_{\varphi} \left( h_t  \right) \right|_{t=0} 
&= \left. \sqrt{-1}  \int_{X} \text{tr}\,  \left(  v_{t} \cdot \left( 
\bar{\partial} D' v_t \wedge \omega^2 
+ [ \varphi , [ \bar{\varphi}^{h_t} , v_t ]] \right) \right)  \right|_{t=0} \\
&= \left. || D' v_t ||^{2}_{L^2} \right|_{t=0}  
\left. + || [ \bar{\varphi}^{h_t}, v_{t}] ||^{2}_{L^2} 
\right|_{t=0} . \\
\end{split}
\end{equation*}
Hence,  $h_{0}$ is at least a local minimum of $D_{\varphi}$. 
We then consider an arbitrary element in $\text{Herm}^{+} (E)$ and join
 it to $h_{0}$ by a geodesic $h_{t}$. 
Since $\partial_{t} v_{t} = 0$ if $h_{t}$ is a geodesic (see \cite[p.~204]{Ko}), 
from the same computation above, 
we obtain  
\begin{equation}
\frac{d^2}{dt^2} 
 D_{\varphi} \left( h_t  \right) 
=  || D' v_t ||^{2}_{L^2}  
 + || [ \bar{\varphi}^{h_t}, v_{t}] ||^{2}_{L^2} . 
\label{eq:dsd}
\end{equation}
This implies Proposition \ref{prop:convex}. 
\end{proof}

As in the case of the Hermitian--Einstein problem \cite[Prop.6
(iii)]{D}, 
the functional $D_{\varphi} (h, k)$ is the integration of Bott--Chern
forms as stated below. 
This can be proved for example by a similar argument to that by
Bradlow--Gomez \cite[Appendix]{BG} for the Higgs bundle case, 
so we here omit the details.

\begin{proposition}
Let $h, h' \in \text{Herm}^{+} (E)$. 
Then 
\begin{equation*}
\begin{split}
\sqrt{-1} \partial \bar{\partial} 
 &\left( Q_{2} ( h , h') 
  -\frac{\lambda (E)}{2} Q_{1} (h ,h') \wedge \omega \right) \\
& \quad = - \frac{1}{2} 
\text{tr} \, \left( \left( 
\sqrt{-1} (F_{h} +  \Lambda [ \varphi , \bar{\varphi}^{h}] )
  - \frac{\lambda (E)}{2} Id_{E} \, \omega \right)^2 \right) \\
& \qquad \qquad 
 + \frac{1}{2} \text{tr}\, 
 \left( \left( \sqrt{-1} ( F_{h'} +  \Lambda 
 [\varphi , \bar{\varphi}^{h'} ] ) 
 - \frac{\lambda (E)}{2} Id_{E} \, \omega \right)^2 \right). 
\end{split}
\end{equation*} 
\label{prop:sec}
\end{proposition}

\subsection{The Hitchin equation on compact Riemann surfaces}
\label{sec:Hit}

In this section, we briefly describe the Hitchin equation \cite{Hit} on
compact Riemann surfaces and its generalization.

\paragraph{The equations.} 
Let $\Sigma$ be a compact Riemann surface, and let $E$ a Hermitian vector
bundle of rank $r$ over $\Sigma$. 
We consider the following equations for a pair $(A , \Phi) \in \mathcal{A}_{E}
\times \Omega^{1,0} (\Sigma , \text{End}\, (E))$, 
where $\mathcal{A}_{E}$ is the set of all connections on
$E$.  
\begin{gather*}
 \bar{\partial}_{A} \Phi = 0 , \quad 
F_{A}  +  [ \Phi , 
 \bar{\Phi}] = \sqrt{-1} \frac{\lambda (E)}{2} Id_{E} \, \omega, 
\end{gather*}
where $F_{A}$ is the curvature of $A$, and 
$\lambda (E) := 2 \pi c_{1} (E) / r [ \omega]$. 
The above equations are called the {\it Hitchin equations}.

\paragraph{Stability.} 
For a pair $(E, \Phi)$ consisting of a holomorphic vector bundle $E$ over a
compact Riemann surface $\Sigma$ and a holomorphic section $\Phi$ of
$\text{End}\, (E) \otimes K_{\Sigma}$, 
(semi-)stability of it can be defined in the same style of Definition
\ref{def:stable}. 
Hitchin \cite{Hit} proved the 
existence of a solution to the equations
can be deduced from the assumption of stability (see also
\cite{Si}).

\paragraph{$\mathbf{L}$-twisted Hitchin equations.}

There is a generalization of the Hitchin equations on a compact Riemann
surface $\Sigma$ by Lin \cite{L} (see also \cite{GR}), in which one takes a line
bundle $L$ on $\Sigma$ instead of $K_{\Sigma}$. 
Namely, one considers the following equations on $X$
for a pair $(A, \Phi)$ consisting
of a connection $A$ and a section $\Phi$ of
$\text{End}\, (E) \otimes L$. 
\begin{equation*}
\bar{\partial}_{A} \Phi = 0 , \quad 
\Lambda F_{A}  + \sigma( [ \Phi , 
 \bar{\Phi}] ) -  \frac{i \lambda (E)}{2} Id_{E} \, = 0 , 
\label{eq:tHit}
\end{equation*}
where $\sigma$ is the contraction of sections of $\text{End}\, (E)
\otimes L \otimes L^{*}$ to those of $\text{End}\, (E)$. 
We call these the {\it $L$-twisted Hitchin equations}.
This was further generalized by \'{A}lvarez-C\'{o}nsul and
Garc\'{\i}a-Prada \cite{AG} 
(see also \cite{BGM}).

As in the case of the Vafa--Witten equations described in the previous
section,   
a similar functional $D_{\Phi} (h,k)_{\Sigma}$ for a pair of Hermitian metrics
$h$ and $k$ 
  with properties
such as Propositions \ref{prop:crit} and \ref{prop:convex} 
can be also defined for the $L$-twisted Hitchin equations as well.

Since the existence of a solution to the $L$-twisted Hitchin equations 
under the assumption of stability with $L$-valued
operator was proved by Lin \cite{L} 
(this was further generalized by \'{A}lvarez-C\'{o}nsul and
Garc\'{\i}a-Prada \cite{AG}, 
see also \cite{BGM}), thus we have the following.

\begin{proposition}
Let $E$ be a holomorphic vector bundle on compact Riemann surface
 $\Sigma$, and let $L$ be a line bundle on $X$. 
Let $\Phi$ be a holomorphic section of $\text{End}\, (E) 
\otimes L$. 
If $(E, \Phi)$ is semi-stable with $L$-valued
 operator, 
then for any fixed Hermitian metric $k$ in
 $\text{Herm}^{+}\, (E)$, the set $\{ \mathcal{D}_{\Phi} (h, k)_{\Sigma} , \, 
 h \in \text{Herm}^{+}\, (E) \}$ is bounded below. 
\label{th:Cardona}
\end{proposition}

We use this in the proof of
Theorem \ref{ith:exist}.

\subsection{Proof of Theorem \ref{ith:exist}}
\label{sec:exist}

In this section, we prove Theorem \ref{ith:exist}. 
We follow a proof for the Hermitian--Einstein case, given by Donaldson
\cite{D} 
(see also \cite[Chap.VI]{Ko}). 
We consider the following evolution equation, which is the gradient flow
for the functional $D_{\varphi} (h ,k)$. 
\begin{equation}
 \partial_{t} h_{t} 
 = - \left( i ( \Lambda F_{h_{t}} + *  [ \varphi,
      \bar{\varphi}^{h_{t}} ] ) 
    - \frac{\lambda (E)}{2} h_{t} \right) . 
\label{eq:hf} 
\end{equation}
The above evolution equation \eqref{eq:hf} has a
unique smooth solution for $0 \leq t < \infty$. 
One can see this by using a similar argument in \cite[Chap.VI]{Ko}.

We then prove the following as in the Hermitian--Einstein case. 
\begin{proposition}
Let $h_{t} \, ( 0 \leq t < \infty)$ be a 1-parameter family in
 $\text{Herm}^{+} \, (E)$, which satisfies \eqref{eq:hf}.
Then
\begin{enumerate}
\item[$(i)$] $\displaystyle{\frac{d}{dt} D_{\varphi} (h_{t} ,k)} =  
  - \left\| i ( \Lambda F_{h_{t}} + *  [ \varphi,
	     \bar{\varphi}^{h_{t}}] ) 
   - \frac{\lambda (E)}{2}  Id_{E} \right\|_{L^2}^2  \leq 0 $. 
\item[$(ii)$] $\displaystyle{\max_{X} \left| 
 i ( \Lambda F_{h_{t}} + * 
[ \varphi, \bar{\varphi}^{h_{t}}]) 
   - \frac{\lambda (E)}{2}  Id_{E}  \right|^2}$ is a monotone decreasing function of $t$. 
\item[$(iii)$] If $D_{\varphi} (h_{t} ,k)$ is bounded from below, then 
$$
\max_{X} \left| i ( \Lambda F_{h_{t}} + * 
[ \varphi,
	     \bar{\varphi}^{h_{t}} ] ) 
   - \frac{\lambda (E)}{2}  Id_{E} \right|^2 \to 0
$$ 
as $t \to \infty$. 
\end{enumerate} 
\label{prop:f}
\end{proposition}

\begin{proof} 
The proof here is a modification of that given in \cite[Chap.VI \S 9,
 pp. 224--226]{Ko}. 
Firstly, the above (i) is nothing but Proposition \ref{prop:crit}. 
Namely, as $h_{t}$ satisfies \eqref{eq:hf}, we get 
\begin{equation*}
\begin{split}
&\frac{d}{dt} 
 D_{\varphi} (h_{t} , k) \\
&\quad = - \left( i(  \Lambda F_{h_{t}} + [ \varphi,
 \bar{\varphi}^{h_{t}}] ) 
   - \frac{\lambda (E)}{2}  Id_{E} , 
     i ( \Lambda F_{h_{t}} + * 
[ \varphi, \bar{\varphi}^{h_{t}}]) 
   - \frac{\lambda (E)}{2}  Id_{E}   \right) \\ 
&\quad = - \left\| i ( \Lambda F_{h_{t}} + * 
[ \varphi, \bar{\varphi}^{h_{t}}] ) 
   - \frac{\lambda (E)}{2}  Id_{E}  \right\|_{L^2}^2 . 
\end{split}
\end{equation*}

To prove (ii), we define the operator $ \Box' s 
 := * ( i D'' D' s \wedge \omega  + [ \varphi, 
 [\bar{\varphi}^{h_{t}} , s ] ] )$. 
One can easily check that $\Box' v_{t} 
 = \partial_{t} \left( 
i (\Lambda F_{h_{t}} + * [ \varphi, \bar{\varphi}^{h_{t}}]  ) 
   - \frac{\lambda (E)}{2}  Id_{E}  \right)$. 
Then, by using the evolution equation \eqref{eq:hf}, we get 
$$(\partial_{t} + \Box') \left( 
i (\Lambda F_{h_{t}} + [ \varphi, \bar{\varphi}^{h_{t}}]  ) 
   - \frac{\lambda (E)}{2}  Id_{E}  \right) =0 . $$
We also have 
\begin{equation*}
\begin{split}
&\Delta \left| i ( \Lambda F_{h_{t}} + * [ \varphi, \bar{\varphi}^{h_{t}} ] )
   - \frac{\lambda (E)}{2}  Id_{E}  \right|^2  
\\  & \qquad = - \partial_{t} \left| 
 i (\Lambda F_{h_{t}} + * [ \varphi, \bar{\varphi}^{h_{t}}]  ) 
   - \frac{\lambda (E)}{2} Id_{E}  \right|^2 \\
 & \qquad \qquad -2 \left| D'' \left(
 i ( \Lambda F_{h_{t}} + *  [ \varphi, \bar{\varphi}^{h_{t}}] ) 
   - \frac{\lambda (E)}{2} Id_{E}  \right) \right|^2 .\\ 
\end{split}
\end{equation*}
Hence, we get 
\begin{equation*}
\begin{split} 
(\partial_{t} + \Delta ) 
& \left|  i ( \Lambda F_{h_{t}} + * [ \varphi,
 \bar{\varphi}^{h_{t}}] ) 
   - \frac{\lambda (E)}{2}  Id_{E} \right|^2 \\ 
 &= -2 \left| 
i ( \Lambda F_{h_{t}} + * [ \varphi, \bar{\varphi}^{h_{t}}] ) 
   - \frac{\lambda (E)}{2}  Id_{E} \right|^2 \leq 0 .  
\end{split}
\end{equation*}
Thus, the maximum principle implies (ii).

Once (i) and (ii) are obtained, then (iii) follows from a similar 
 argument as in \cite[Chap.VI \S 9, pp 225--226]{Ko}. 
Namely, by using a maximum principle argument, one bounds 
$\max_{X} | i ( \Lambda F_{h_{t}} + * [ \varphi, \bar{\varphi}^{h_{t}}] ) 
   - \frac{\lambda (E)}{2}  Id_{E} |^2$ by 
$ \|  i ( \Lambda F_{h_{t}} + * [ \varphi, \bar{\varphi}^{h_{t}}] ) 
   - \frac{\lambda (E)}{2}  Id_{E}  \|_{L^2}^{2}$. 
One then deduces 
$\| i ( \Lambda F_{h_{t}} + * [ \varphi, \bar{\varphi}^{h_{t}}] ) 
   - \frac{\lambda (E)}{2}  Id_{E} \|_{L^2}^{2} \to 0$ as $t \to \infty$
 from (ii). 
\end{proof}

We then prove that $D_{\varphi} (h,k)$ is bounded below under the 
assumption 
 that $E$ is semi-stable.

Firstly, we recall that the notion of stability has the following 
generalization introduced by Simpson
\cite[\S 3]{Si2}.

\begin{definition} 
(a) 
Let $\mathcal{W}$ be a vector bundle on a compact K\"{a}hler manifold $X$. 
A sheaf $\mathcal{E}$ together with a map $\eta : \mathcal{E} \to
\mathcal{E} \otimes \mathcal{W}$ is called 
{\it a sheaf with $\mathcal{W}$-valued operator} $\eta$. 
(b) 
A  torsion-free sheaf $\mathcal{E}$ with $\mathcal{W}$-valued operator
 $\eta$ is said to be {\it semi-stable} 
if $ \mu ( \mathcal{F} ) \leq \mu ( \mathcal{E} ) $
for any coherent subsheaves $\mathcal{F}$ of $\mathcal{E}$ with 
$\text{rank} (\mathcal{F}) 
< \text{rank} (\mathcal{E})$ and $\eta (\mathcal{F}) \subset
\mathcal{F} \otimes \mathcal{W}$. 
It is called {\it stable} if the strict inequality holds in the
 definition of the semi-stability.  
\label{def:stablew}
\end{definition}

The (semi-)stability of Definition \ref{def:stable} is included in
this generalized notion of (semi-)stability by taking $\mathcal{W}
= K_{X}$. 
This notion was further generalized to {twisted quiver sheaves} 
by \'{A}lvarez-C\'{o}nsul and Garc\'{\i}a-Prada \cite{AG} 
(see also \cite{BGM}).

As mentioned in \cite[pp.37--38]{Si2}, 
we note that the arguments of 
Mehta--Ramanathan
\cite{MR1}, \cite{MR2} indicate the following.

\begin{proposition}[\cite{Si2}, Prop.~3.6]
If $\mathcal{E}$ is a torsion-free (semi-)stable sheaf on $X$ with
 $\mathcal{W}$-valued operator, then there exists a positive integer $m$
 such that, for a generic smooth curve $D \subset X$ in a linear system 
$| \mathcal{O}_{X} (m) |$, 
$\mathcal{E}|_{D}$ is a (semi-)stable sheaf with
 $\mathcal{W}|_{D}$-valued operator. 
\label{prop:MRS}
\end{proposition}

From this, we get the following.

\begin{Corollary}
Let $( \mathcal{E} , \varphi)$ be a torsion-free (semi-)stable pair 
 on $X$ in the sense of Definition \ref{def:stable}. 
Then, 
there exists a positive integer $m$
 such that, for a generic smooth  curve $D \subset X$ in a linear system 
$| \mathcal{O}_{X} (m) |$, 
$( \mathcal{E}|_{D} , \varphi|_{D} )$ is a (semi-)stable pair 
with $(K_{D} \otimes \mathcal{O}_{D} (-D))$-valued operator on $D$. 
\label{prop:MR}
\end{Corollary}

\begin{proof}
This just follows from Proposition \ref{prop:MRS} and the adjunction
 formula. 
In fact, 
from Proposition \ref{prop:MRS}, 
there exists a positive integer $m$ such that, for a generic smooth curve $D
 \subset X$ in
 $| \mathcal{O}_{X} (m) |$, $\mathcal{E}|_{D}$ is a (semi-)stable sheaf with $K_{X}
 |_{D}$-valued operator. 
On the other hand, from the adjunction formula, we have 
$K_{X} |_{D} = K_{D} \otimes  \mathcal{O}_{D} (-D)$. 
Hence, the assertion holds. 
\end{proof}

We now describe the behaviour of the functional $D_{\varphi} (h,k)$
when it is restricted to a smooth curve $D \subset X$. 
The notation we use is that $D_{\varphi} (h,k)_{X}$ is the functional for the
Vafa--Witten equations on $E$, 
and $D_{\varphi} (h,k)_{D}$ is that for the $L$-twisted Hitchin
equations on $E|_{D}$, where we take $L = K_{D} \otimes \mathcal{O}_{D} (-D)$.

\begin{proposition}
Let $E$ be a holomorphic vector bundle on a smooth projective surface
 $X$, and let $\varphi$ be a holomorphic section of $\text{End}\, (E)
 \otimes K_{X}$. 
Let $D$ be a smooth curve of $X$ 
such that the line bundle $F$ defined by $D$ is ample. 
We use a positive closed $(1,1)$-form $\omega$ representing the Chern
 class of $F$ as a K\"{a}hler form for $X$.  
Then for a fixed Hermitian metric $k$ of $E$ and for all Hermitian
 metrics $h$ of $E$, we have 
\begin{equation}
D_{\varphi} (h ,k )_{X} 
\geq  D_{\varphi} (h, k)_{D} 
- C \left( \max_{X} \left| i ( \Lambda F_{h} 
   + * [ \varphi , \bar{\varphi}^{h}] ) - \frac{\lambda (E)}{2} Id_{E}   \right|^2 \right) 
   -C',
\end{equation}
where $C$ and $C'$ are positive constants. 
\label{prop:fr}
\end{proposition}

\begin{proof}
Firstly, we recall the Poincar\'{e}--Lelong formula as in
 \cite[pp.12--13]{D}, \cite[Chap.VI \S10]{Ko}. 
We follow notations in \cite[Chap.VI \S10]{Ko}. 
Let $M$ be a smooth projective variety , and let $V$ be a closed
 hypersurface of $M$. 
Let $F$ be an ample line bundle on $M$ with 
a global holomorphic section $s$ such that $s^{-1} (0) = V$. 
We take a $C^{\infty}$ positive section $a$ of $F \otimes \bar{F}$. 
As $F $ is ample, we can take such $a$ so that $\omega := \frac{i}{2
 \pi} \partial \bar{\partial} \log \, a$ is positive. 
We use this as a K\"{a}hler form on $M$, and the restriction 
$\omega_{V}$ as a K\"{a}hler
 form on $V$. 
Put $f := |s|^2 /a$. Then, for all $(n-1,n-1)$-form $\eta$ on $M$, we
 have the the following as current (see \cite[Chap.VI \S 10]{Ko}). 
\begin{equation}
\frac{i}{2 \pi} \int_{M} 
 ( \log f) \partial \bar{\partial} \eta 
 = \int_{V} \eta - \int_{M} \eta \wedge \omega . 
\label{eq:current}
\end{equation}
We now set $\eta = Q_{2} (h ,k) - \frac{\lambda (E)}{2} 
Q_{1} (h, k) \wedge \omega$ in \eqref{eq:current}, and get  
\begin{equation*}
\begin{split}
D_{\varphi} 
 (h ,k)_{X} 
 &=  \int_{X} \left( Q_{2} (h ,k) - \frac{\lambda (E)}{2} Q_{1} (h,k)  \wedge \omega
 \right) \wedge \omega  \\
 &= \int_{D} \left( Q_{2} (h, k ) - \frac{\lambda (E)}{2} Q_{1} (h,k)
 \wedge \omega \right)  \\
  & \qquad  - \frac{i}{2 \pi} \int_{X} (\log f ) \partial \bar{\partial} 
 \left( Q_{2} ( h,k) - \frac{\lambda (E)}{2} Q_{1} (h,k)   \wedge \omega
 \right) . 
\end{split}
\end{equation*}

We then use Proposition \ref{prop:sec} to obtain 
\begin{equation}
\begin{split}
D_{\varphi} 
(h , k)_{X} &= 
 D_{\varphi} (h , k)_{D} \\
 & + \frac{1}{4 \pi} 
\int_{X} ( \log f) \left( \text{tr} 
 \left( i (  F_{h} 
   + \Lambda  [ \varphi , \bar{\varphi}^{h}]  ) - \frac{\lambda (E)}{2}
 Id_{E} \, \omega  \right)^2 \right) \\
 & \qquad  - \frac{1}{4 \pi} 
\int_{X} (\log f) \left( \text{tr}  
  \left( i ( F_{k} 
   + \Lambda  [ \varphi , \bar{\varphi}] ) - \frac{\lambda (E)}{2} Id_{E} \,
 \omega  
  \right)^2 \right) . 
\end{split}
\label{eq:ac}
\end{equation}
Note that the last term in the right hand side of \eqref{eq:ac} is
bounded by a constant as the Hermitian metric $k$ is fixed. 
We then estimate the second term in the right hand side of
\eqref{eq:ac}. 
As in \cite{Ko}, we use the primitive decomposition to obtain 
 $\text{End}(E)$-valued $(1,1)$-form $S$ with 
$F_{h} = (\Lambda F_{h}) \omega + S$ and $S \wedge \omega =0$. 
Then we get 
\begin{equation*}
\begin{split}
& \left( i (  F_{h} + \Lambda  [ \varphi , \bar{\varphi}^{h}] ) 
 - \frac{\lambda (E)}{2} Id_{E} \, \omega \right)^2 \\
& \qquad  \qquad 
 = \left( i (  \Lambda F_{h} + *  [ \varphi , \bar{\varphi}^{h}]
 ) \omega +
 iS  
 - \frac{\lambda (E)}{2} Id_{E} \omega \right)^2 \\ 
& \qquad \qquad \qquad 
 = \left( i (  \Lambda F_{h} + *  [ \varphi , \bar{\varphi}^{h}]
 ) 
 - \frac{\lambda (E)}{2} Id_{E}  \right)^2 \omega^2 - S \wedge S . \\ 
\end{split}
\end{equation*}
Hence, we obtain  
\begin{equation*}
\begin{split}
& (\log f)  \left(  \text{tr}    
\left( i (  F_{h} + \Lambda  [ \varphi , \bar{\varphi}^{h}] ) 
 - \frac{\lambda (E)}{2} Id_{E} \, \omega \right)^2 \right) \\
& \qquad  
 = (\log f) \left( \text{tr} \left( i (  \Lambda F_{h} + *  
[ \varphi , \bar{\varphi}^{h}] ) \omega 
 - \frac{\lambda (E)}{2} Id_{E} \, \omega \right)^2  -
 \text{tr}\, (S \wedge S) \right). \\
\end{split}
\end{equation*}
As in \cite[p.~233]{Ko}, $\text{tr}\, (S \wedge S) \geq 0$, and 
we can also assume that $f = |s|^2 /a \leq 1$, thus, 
\begin{equation*}
\begin{split}
& (\log f )  \left(  \text{tr}  
\left( i (  F_{h} + \Lambda [ \varphi , \bar{\varphi}^{h}] ) 
 - \frac{\lambda (E)}{2} Id_{E} \, \omega \right)^2 \right) \\
& \qquad \qquad  
 \geq (\log f) \left|  i (  \Lambda F_{h} + *  [ \varphi ,
 \bar{\varphi}^{h}] ) 
 - \frac{\lambda (E)}{2} Id_{E}  \right|^2 \omega^2 .\\
\end{split}
\end{equation*}
The proposition follows from this and \eqref{eq:ac}.  
\end{proof}

We now take a smooth curve  $D$ of $X$ in Corollary \ref{prop:MR}. 
Then, by Proposition \ref{prop:fr}, 
for a given Hermitian metric $k$ there exists positive constants $C$
and $C'$ such that 
\begin{equation*}
\begin{split}
&D_{\varphi} (h_{t} , k)_{X} \\
& \quad  \geq D_{\varphi} (h_{t} , k)_{D} - C \left( 
 \max_{X} \left|  
  i ( \Lambda F_{h_{t}} 
   + *  [ \varphi , \bar{\varphi}^{h_{t}}]) - \frac{\lambda (E)}{2} Id_{E} 
  \right|^2  \right) - C' . \\ 
\end{split}
\end{equation*}

From Proposition \ref{prop:f} (i), we have 
$ D_{\varphi} (h , k)_{X} 
 \geq D_{\varphi} (h_{t} , k)_{X} $ for all $t >0$. 
On the other hand, 
from Proposition \ref{prop:f} (ii), 
there exists $t_{1} >0$ such that 
$\max_{X} \left| 
i ( \Lambda F_{h_{t}} 
   + * [ \varphi , \bar{\varphi}^{h_{t}}] ) - \frac{\lambda (E)}{2} Id_{E} 
\right| < 1 $ 
for $t \geq t_{1}$. 
Thus, we get 
$$ D_{\varphi} (h,k)_{X} 
 \geq D_{\varphi} (h_{t} ,k)_{D} -C -C' $$
for $t \geq t_{1}$ by Proposition \ref{prop:fr}. 
From Proposition \ref{th:Cardona}, 
$D_{\varphi} (h_{t} , k)_{D}$ is bounded  below, thus so is  
$D_{\varphi} (h,k)_{X}$.

Once the lower bound of the functional is obtained, 
the argument originally implemented by Donaldson \cite[\S 3]{D} works 
for the Vafa--Witten case as we describe it below.

Firstly, the Uhlenbeck type weak compactness theorem for the
Vafa--Witten equations by Mares \cite[\S 3.3]{BM} applies 
to give an $L^{p}_{1}$ limit along the flow away from a finite set of
points. 
On the other hand, since we have the lower bound for the functional
$D_{\varphi} ( \cdot)$, 
$C^{0}$-norm of $i ( \Lambda F_{h_{t}} 
   + * [ \varphi , \bar{\varphi}^{h_{t}}] ) - \frac{\lambda (E)}{2} Id_{E} 
$ converges to zero by Proposition \ref{prop:f} (iii). 
Hence the above $L^{p}_{1}$ weak limit satisfies the Vafa--Witten
equations. 
We then invoke the removal singularity theorem for the Vafa--Witten
equation by Mares \cite[\S 3.4]{BM} to obtain a bundle $E'$ and a
solution to the Vafa--Witten equations across the singular set. 
This bundle is semi-stable in the sense of Definition \ref{def:stable}
as it admits a solution to the Vafa--Witten equations. 
One then constructs a non-zero holomorphic map between $E$ and $E'$. 
From the assumption that $E$ is stable, this map is isomorphism. 
Hence, we a obtain a solution to the Vafa--Witten equations on $E$. 
The uniqueness of the solution follows from the
convexity of the functional $D_{\varphi} ( \cdot)$. 
\qed


\begin{flushleft}
National Center for Theoretical Sciences (South), 
National Cheng Kung University, Tainan 701, Taiwan \\
yu2tanaka@gmail.com
\end{flushleft}

\end{document}